\documentclass[compress, preprint,9pt]{elsarticle}

\usepackage[a4paper, total={6in,8in}]{geometry}
\usepackage{amsthm,amssymb,amsmath,graphicx,amscd,mathrsfs}
\usepackage{color,xcolor,comment,dsfont,pifont,enumitem}
\usepackage{longtable,graphicx,float,amsfonts,setspace}
\usepackage{hyperref,multirow,placeins,algorithm,algorithmicx}

\numberwithin{equation}{section}

\newtheorem{lemma}{Lemma}[section]
\newtheorem{defi}{Definition}[section]
\newtheorem{theo}{Theorem}[section]
\newtheorem{coro}{Corollary}[section]
\newtheorem{example}{Example}[section]
\newtheorem{remark}{Remark}[section]

\def\b0{\boldsymbol{0}}
\newcommand{\E}{\mathcal{E}}
\newcommand{\T}{{\mathcal{T}}}
\newcommand{\Q}{{\mathcal{Q}}}

\newcommand{\sumT}{\sum_{T\in\mathcal{T}_h}}     %new
\newcommand{\bn}{\mathbf{n}}

\newcommand{\di}{\text{div}}
\newcommand{\eps}{\varepsilon}

\newcommand{\A}{\mathcal{A}}
\newcommand{\Real}{\mathbb{R}}

\newcommand{\la}{\langle}
\newcommand{\ra}{\rangle_{\partial T}}
\newcommand{\ran}{\rangle}
\allowdisplaybreaks[4]

\allowdisplaybreaks % For long formula

\begin{document}

\begin{frontmatter}

\title{
	The weak Galerkin method for a class of Gross-Pitaevskii type eigenvalue problems
	}

\author[mymainaddress]{Wei Lu}
\ead{luwei24@mails.jlu.edu.cn}

%\author[mysecondaddress]{Hehu Xie}
%\ead{hhxie@lsec.cc.ac.cn}

\author[mymainaddress]{Qilong Zhai\corref{mycorrespondingauthor}}
\cortext[mycorrespondingauthor]{Corresponding author}
\ead{zhaiql@jlu.edu.cn}

\address[mymainaddress]{School of Mathematics, Jilin University, Changchun 130012, Jilin, China} 
\address[mysecondaddress]{LSEC and Institute of Computational
	Mathematics and Scientific/Engineering Computing, Academy of
	Mathematics and Systems Science, Chinese Academy of Sciences,
	Beijing 100190, China}
%\address[mythirdaddress]{School of Mathematics, Jilin University, Changchun 130012, Jilin, P. R. China} 
\begin{abstract}
	This paper aims to employ the weak Galerkin method to solve a class of nonlinear eigenvalue problems. We proved the weak Galerkin scheme produces lower bound for the energy. Moreover, by the post-processing technique, we obtain lower bound for the ground state eigenvalue. Finally, numerical experiments are provided to validate the theoretical analysis.
\end{abstract}

\begin{keyword}
 


Weak Galerkin method, Nonlinear eigenvalue problem, Error analysis, Lower bounds.

\MSC[2008] 65N15 \sep 65N25 \sep 65N30.

\end{keyword}

\end{frontmatter}

 %%\linenumbers


\section{Preliminaries}In this section, we state some notations and preliminaries. 

The standard Sobolev space notations are used in this paper. Let $\Omega\subset\Real^2$ be a bounded domain with $\partial\Omega$, and $H^m(\Omega)$ be the Sobolev space. The notations $(\cdot, \cdot)_{m,D}$ and $||\cdot||_{m,D}$ are used as inner-product and norms on $H^m(D)$, if the region $D$ is an edge of some elements, we use $\langle\cdot,\cdot\rangle _{m,D}$ instead of $(\cdot,\cdot)_{m,D}$. For simplicity, We shall drop the subscript when $m=0$ or $D=\Omega$.

Consider the following variational models of the form
\begin{eqnarray}\label{1.1}
	\inf\{E(v):v\in H_0^1(\Omega),\displaystyle\int_\Omega v^2 d\Omega=1\},
\end{eqnarray}
where $\Omega\subset\Real^2$ is a convex bounded domain, and the energy function $E$ is defined as
\begin{eqnarray}\label{E}
	E(v)=\frac12 a(v,v)+\frac12\displaystyle\int_\Omega F(v^2)d\Omega,
\end{eqnarray}
where
\begin{eqnarray*}
	a(u,v)=(\A\nabla u,\nabla v)+(Vu,v),\quad\forall u,v\in H^1(\Omega),
\end{eqnarray*}
with $\A(x)\in (L^\infty(\Omega))^{2\times2}$ is symmetric, and 
%$V(x)\in L^p(\Omega)(1<p\le\infty)$
$V(x)\in L^2(\Omega)$ is the non-negative potential.

Next, we introduce the following assumptions.

\textbf{Assumption (A1).}\emph{There exists a constant $\alpha$ such that $\displaystyle\inf_{x\in \Omega}\xi^T\A(x)\xi\ge\alpha\xi^T\xi$, $\forall \xi\in R^2$, $\forall x\in\Omega$ a.e..
}

\textbf{Assumption (A2).}\emph{$F(t)\in C^1[0,+\infty)\cap C^2(0,+\infty)$, and $\displaystyle\inf_{t>0}F''(t)>0$.
}

\textbf{Assumption (A3).}\emph{There exists $q\in [0,2)$ and a constant $C$ such that $|F'(t)|\le C(1+t^q)$, $\forall t\ge0$.
}

\textbf{Assumption (A4).}\emph{$F''(t)t$ locally bounded in $[0,+\infty)$.
}

\textbf{Assumption (A5).}\emph{There exists $1<r\le 2$ and $0\le s\le 3-r$ such that: $\forall R>0$, there holds $\forall t_1\in(0,R)$ and $\forall t_2\in\Real$ that
	\begin{eqnarray*}
		|F'(t_2^2)t_2-F'(t_1^2)t_2-2F''(t_1^2)t^2(t_2-t_1)|\le C(1+|t_2|^s)|t_2-t_1|^r,
	\end{eqnarray*}
	where C is a positive constant depend on R.
}

For simplicity, we denote by $f(t)=F'(t)$. Then, under Assumption (A2), $E$ is twice differentiable on $H^1(\Omega)$ and we have
\begin{eqnarray*}
	E'(v)=A^vv,\quad \la E''(u)v,w\ran=\la A^uv,w\ran+2(f'(u^2)u^2v,w),\quad\forall u,v,w\in H^1(\Omega),
\end{eqnarray*}
where
\begin{eqnarray*}
	A^vw=-\di(\A\nabla\cdot w)+Vw+f(v^2)w,\quad\forall v,w\in H^1(\Omega).
\end{eqnarray*}
It is easy to check, for each $u\in L^2(\Omega)$, $A^u$ defines a self-adjoint operator and satisfies
\begin{eqnarray*}
	\la A^uv,w\ran=(\A\nabla v,\nabla w)+(Vv,w)+(f(u^2)v,w),\quad\forall v,w\in H_0^1(\Omega).
\end{eqnarray*}

\begin{lemma}\label{lem1}
	Under Assumptions (A1)-(A3), \eqref{1.1} has a unique (up to sign) minimizer $u$. Moreover, there exists $\lambda\in\Real$ such that $(u,\lambda)$ is the first eigenpair of the following nonlinear eigenvalue problem:
	\begin{equation}\label{1.3}
		\left\{\begin{array}{rcl}
			\la A^uu,v\ran&=&\lambda (u,v),\quad\forall v\in H_0^1(\Omega),\\
			\|u\| &=&1.
		\end{array}\right.
	\end{equation}
	Besides, $u\in C(\bar{\Omega})\cap L^\infty(\Omega)$ is positive in $\Omega$, and $\lambda$ is a simple eigenvalue of $A^u$.
\end{lemma}
\begin{proof}
	See Lemma 2 in \cite{Eric2009}.
\end{proof}

\section{Weak Galerkin discretization}In this section, we introduce the standard WG scheme for nonlinear eigenvalue problem \eqref{1.3}.

Let $\T_h$ be a partition of the domain $\Omega$, and the elements
in $\T_h$ are polygons satisfying the regular assumptions specified
in \cite{Wang2014a}. Let $\E_h$ be the edges in $\T_h$, and
$\E_h^0$ denotes by the interior edges $\E_h \backslash\partial\Omega$. For each element $T\in\T_h$, $h_T$ represents the diameter of $T$, and $h=\max\limits_{T\in\T_h} h_T$ denotes the mesh size. For simplicity, we use $a\lesssim b$  and $a\gtrsim b$ to represent $a\le Cb$ and $a\ge Cb$, respectively, where $C$ is a constant independent of mesh size $h$.

%According to \cite{Wang2013a}, the following trace inequality and inverse inequality hold true.
\begin{lemma}
	For any element  $T\in\T_h$, there holds the trace inequality
	\begin{eqnarray}\label{trace}
		\|v\|_{\partial T}^2\lesssim h_T^{-1}\|v\|_T^2+h_T\|\nabla v\|_T^2,\quad\forall v\in h^1(T).
	\end{eqnarray}
\end{lemma}
\begin{proof}
	See Lemma A.3 in \cite{Wang2014a}.
\end{proof}

%\begin{lemma}
%	For any element $T\in\T_h$, there holds the inverse inequality
%	\begin{eqnarray}\label{inverse}
	%		\|\nabla v\|_T\lesssim h_T^{-1}\|v\|_T,\quad\forall v\in P_k(T).
	%	\end{eqnarray}
%\end{lemma}

Then, we introduce the WG discretization for the eigenvalue problem \eqref{1.3}. For a given integer $k\ge1$, the WG finite element space is defined by
\begin{eqnarray*}
	V_h= \big\{ v=\{ v_0, v_b\}: v_0|_T\in  P_k(T),  v_b|_e\in  P_{k-1}(e),\ \forall T\in\T_h,\ e\in\E_h,\text{ and }  v_b|_{\partial\Omega}= 0\big\},
\end{eqnarray*}

%Define the sum space $V=V_h+ H_0^1(\Omega)$. For each $ v\in V$, we define its weak gradient $\nabla_w v$ as follows.
% \begin{defi}
	%	$\nabla_w v|_T$ is the unique polynomial in $[P_{k-1}(T)]^{2}$ satisfying
	%	 \begin{eqnarray*}
		%		(\nabla_w v,q)_T=-(v_0,\nabla\cdot q)_T+\la v_b,q\cdot\n
		%		\ra,\quad\forall q\in [P_{k-1}(T)]^{2},
		%	\end{eqnarray*}
	%	where $\n$ denotes the outward unit normal vector.
	%\end{defi}
	
	For each $ v_h\in V_h$, we define its weak gradient $\nabla_w v_h$ as follows.
	\begin{defi}
		For each $ v_h\in V_h$, $\nabla_w v_h|_T$ is the unique polynomial in $[P_{k-1}(T)]^{2}$
		satisfying
		\begin{eqnarray*}
			(\nabla_w v_h,q)_T=-(v_0,\nabla\cdot q)_T+\la v_b,q\cdot\bn
			\ra,\quad\forall q\in [P_{k-1}(T)]^2,
		\end{eqnarray*}
		where $\bn$ denotes the outward unit normal vector.
	\end{defi}
	
	For each $T\in\T_h$, let $Q_0$ denotes the $L^2$ projection from $L^2(T)$ onto $ P_k(T)$, $\Q_h$ denotes the $L^2$ projection from $[L^2(T)]^{2}$ onto $[P_{k-1}(T)]^{2}$.
	For each $e\in\E_h$, let $Q_b$ denotes the $L^2$ projection from $L^2(e)$ onto $ P_{k-1}(e)$ for each $e\in\E_h$.
	Combining $Q_0$ and $Q_b$ together, we define $Q_h=\{Q_0,Q_b\}$,
	which is a projection onto $V_h$.
	
	For the projection operators defined above, we have the following results.
	\begin{lemma}\label{lem3}
		There holds the following commutative property on each $T\in\T_h$:
		\begin{eqnarray*}
			\nabla_wQ_hv=\Q_h\nabla v, \quad\forall v\in H^1(T).
		\end{eqnarray*}
	\end{lemma}
	\begin{proof}
		See Lemma 5.1 in \cite{Mu2013c}.
	\end{proof}
	
	%Refer to \cite{Mu2013c}, the following estimates hold true for the projection operators.
	\begin{lemma}
		For any $v\in H^{k+1}(\Omega)$ and $\tau\in [H^k(\Omega)]^2$, $s=0,1$, we have
		\begin{eqnarray}
			\sumT h^{2s}\|v-Q_hv\|_{s,T}^2&\lesssim h^{2(k+1)}\|v\|_{k+1}^2,\label{pro1}\\
			\sumT h^{2s}\|\tau-\Q_h\tau\|_{s,T}^2&\lesssim h^{2k}\|\tau\|_k^2.\label{pro2}
		\end{eqnarray}
	\end{lemma}
	\begin{proof}
		See Lemma 4.1 in \cite{Wang2014a}.
	\end{proof}
	
	In order to give the WG scheme, we define two bilinear forms on $V_h$.
	\begin{eqnarray}
		s( v_h, w_h)=&\sumT h_T^{-1+\eps}\langle  Q_bv_0- v_b, Q_bw_0- w_b\rangle_{\partial T},\label{2.1}\\
		a_h( v_h, w_h)=&(\A\nabla_wv_h,\nabla_ww_h)+(Vv_h,w_h)+s(v_h,w_h),\label{2.2}
	\end{eqnarray}
	where $v_h=\{v_0,v_b\}$, $w_h=\{w_0,w_b\}\in V_h$, $\eps\in(0,1)$ is a small constant.
	
	In addition, for each $u_h\in V_h$, we introduce the operator $A_h^{u_h}$ by
	\begin{eqnarray}\label{2.3}
		\la A_h^{u_h}v_h, w_h\ran=a_h( v_h, w_h)+(f(u_h^2)v_h,w_h), \quad\forall v_h,w_h\in V_h,
	\end{eqnarray}
	and the discrete energy function is given by
	\begin{eqnarray}\label{Eh}
		E_h(v_h)=\frac12 a_h(v_h,v_h)+\frac12\displaystyle\int_\Omega F(v_h^2)d\Omega, \quad\forall v_h\in V_h.
	\end{eqnarray}
	Similarly, $E_h$ is twice differentiable on $V_h$ and satisfies
	\begin{eqnarray}\label{E''h}
		\la E_h''(u_h)v_h,w_h\ran=\la A_h^{u_h}v_h,w_h\ran+2(f'(u_h^2)u_h^2v_h,w_h),\quad\forall u_h, v_h,w_h\in V_h.
	\end{eqnarray}
	
	Now, we are ready to give the weak Galerkin scheme for the nonlinear eigenvalue problem \eqref{1.3}.
	\begin{algorithm}[ht!]
		\caption{The weak Galerkin scheme for \eqref{1.3}.}
		Find $ u_h\in V_h$ and $\lambda_h\in\Real$ such that $\| u_h\|=1$ and
		\begin{eqnarray}\label{2.4}
			\la A_h^{u_h} u_h, v_h\ran=\lambda_h(u_h,v_h),\quad\forall  v_h\in V_h.
		\end{eqnarray}
	\end{algorithm}

\section{Error analysis}In this section, we show the error analysis for the WG scheme \eqref{2.4}.

For the aim of analysis, we introduce
\begin{eqnarray*}
	\|v_h\|_{1,h}^2=\sumT\|\nabla v_0\|_T^2+\sumT h_T^{-1}\|  Q_bv_0- v_b\|_{\partial T}^2,\quad\forall v_h\in V_h.
\end{eqnarray*}
It is easy to check $\|\cdot\|_{1,h}$ defines a norm on $V_h$ \cite{Zhai2019c,Wang2017b}.

Consider the following auxiliary problem: Find $\widetilde{u}_h\in V_h$ and $\widetilde{\lambda}_h\in\Real$ such that
\begin{eqnarray}\label{3.3}
	\la A_h^{u} \widetilde{u}_h, v_h\ran=\widetilde{\lambda}_h(\widetilde{u}_h,v_h),\quad\forall  v_h\in V_h,
\end{eqnarray}
where $u$ is the non-negative minimizer of \eqref{1.1}. Then, we have the following estimates.
\begin{lemma}\label{lem7}
	Let $(\widetilde{u}_h,\widetilde{\lambda}_h)$ be the first eigenpair of \eqref{3.3}. Suppose the ground state $u\in H^{k+1}(\Omega)$, then we have
	\begin{eqnarray}
		h^{2k}\|u\|_{k+1}\lesssim\lambda-\widetilde{\lambda}_h&\lesssim h^{2k-2\eps}\|u\|_{k+1},\label{3.4}
		%		|\lambda-\widetilde{\lambda}_h|&\lesssim h^{2k-2\eps}\|u\|_{k+1},\label{3.4}
		\\
		\|u-\widetilde{u}_h\|_{1,h}&\lesssim h^{k-\eps}\|u\|_{k+1},\label{3.5}
		\\
		\|u-\widetilde{u}_h\|&\lesssim h^{k+1-\eps}\|u\|_{k+1},\label{3.6}
	\end{eqnarray}
	when the mesh size h is sufficiently small.
\end{lemma}
\begin{proof}
	See \cite{Zhai2019c}.
\end{proof}

\begin{lemma}
	The bilinear form $a_w(\cdot,\cdot)$ is bounded with respect to $\|\cdot\|_{1,h}$, i.e.
	\begin{eqnarray}\label{3.1}
		|a_w(v_h,w_h)|\lesssim\|v_h\|_{1,h}\|w_h\|_{1,h}, \quad\forall v_h,w_h\in V_h.
	\end{eqnarray}
	%	Furthermore, there holds
	%	\begin{eqnarray}\label{3.2}
		%		a_w(v_h,v_h)\gtrsim h^\eps\|v_h\|_{1,h}^2,\quad\forall v_h\in V_h.
		%	\end{eqnarray}
\end{lemma}
\begin{proof}
	See \cite{Mu2013c,Ye2020,Zhai2019c}.
\end{proof}

\begin{lemma}\label{lem6}
	For any $p\ge1$, there holds the following estimate:
	\begin{eqnarray*}
		\|v_h\|_{L^p}\lesssim\|v_h\|_{1,h}, \quad\forall v_h\in V_h.
	\end{eqnarray*}
\end{lemma}
\begin{proof}
	See Lemma 3 in \cite{Maday2023}.
\end{proof}

\begin{theo}\label{theo1}
	Under Assumptions (A1)-(A3), the following estimate holds true:
	\begin{eqnarray*}
		\left|\la (A_h^u-\widetilde{\lambda}_h)v_h,w_h\ran\right|\lesssim \|v_h\|_{1,h}\|w_h\|_{1,h}, \quad\forall v_h,w_h\in V_h.
	\end{eqnarray*}
\end{theo}
\begin{proof}
	It follows from the definition in \eqref{2.1}-\eqref{2.3} that
	\begin{eqnarray*}
		\la (A_h^u-\widetilde{\lambda}_h)v_h,w_h\ran=a_h(v_h,w_h)+(f(u^2)v_h,w_h)-\widetilde{\lambda}_h(v_h,w_h).
	\end{eqnarray*}
	The first term has been estimated in \eqref{3.1}. For the second term, by Assumption (A3) we get
	\begin{eqnarray*}
		(f(u_h^2)v_h,w_h)\lesssim \|1+|u|^{2q}\|_{L^{3/2}}\|v_h\|_{L^6}\|w_h\|_{L^6}.
	\end{eqnarray*}
	Then, the Sobolev embedding $H^1(\Omega)\hookrightarrow L^6(\Omega)$ implies that
	\begin{eqnarray*}
		(f(u_h^2)v_h,w_h)\lesssim\|v_h\|_{1,h}\|w_h\|_{1,h}.
	\end{eqnarray*}
	For the last term, by Lemma \ref{lem6} we obtain
	\begin{eqnarray*}
		\widetilde{\lambda}_h(v_h,w_h)\lesssim\|v_h\|\|w_h\|\lesssim\|v_h\|_{1,h}\|w_h\|_{1,h}.
	\end{eqnarray*}
	The proof is finished.
\end{proof}

%\begin{coro}
%	$\left|\la (E_h''(u)-\widetilde{\lambda}_h)v_h,w_h\ran\right|\lesssim \|v_h\|_{1,h}\|w_h\|_{1,h}$, $\forall v_h,w_h\in V_h$.
%\end{coro}
%\begin{proof}
%	By the definition \eqref{E''h}, we have
%	\begin{eqnarray*}
	%		\la (E_h''(u)-\widetilde{\lambda}_h)v_h,w_h\ran=\la (A_h^u-\widetilde{\lambda}_h)v_h,w_h\ran+2(f'(u^2)u^2v_h,w_h).
	%	\end{eqnarray*}
%	It follows from Assumption (A4) that
%	\begin{eqnarray*}
	%		\Big|(f'(u^2)u^2v_h,w_h)\Big|\le \|f'(u^2)u^2\|_{L^\infty}\|v_h\|\|w_h\|.
	%	\end{eqnarray*}
%	Then, by Assumption (A4) and Lemma \ref{lem6}, we get
%	\begin{eqnarray*}
	%		\Big|(f'(u^2)u^2v_h,w_h)\Big|\lesssim \|v_h\|_{1,h}\|w_h\|_{1,h},
	%	\end{eqnarray*}
%	which combines with Theorem \ref{theo1} completes the proof.
%\end{proof}

\begin{theo}\label{theo2}
	$\la (A_h^u-\widetilde{\lambda}_h)v_h,v_h\ran\ge 0$, $\forall v_h\in V_h$.
\end{theo}
\begin{proof}
	From \eqref{3.3} we get
	\begin{eqnarray*}
		\la (A_h^u-\widetilde{\lambda}_h)v_h,v_h\ran=\la(A_h^u-\widetilde{\lambda}_h)\left(v_h-(v_h,\widetilde{u}_h)\widetilde{u}_h\right),v_h-(v_h,\widetilde{u}_h)\widetilde{u}_h\ran.
	\end{eqnarray*}
	By Lemma \ref{lem1} and \eqref{3.6}, $\widetilde{\lambda}_h$ is a simple eigenvalue of $A_h^u$ when $h$ is sufficient small. Furthermore, $A_h^u-\widetilde{\lambda}_h$ is coercive on the subspace  span$\{\widetilde{u}_h\}^\perp$. Therefore, we have
	\begin{equation}\label{3.7}
		\begin{aligned}
			\la (A_h^u-\widetilde{\lambda}_h)v_h,v_h\ran&\gtrsim \|v_h-(v_h,\widetilde{u}_h)\widetilde{u}_h\|^2\\
			&=\|v_h\|^2-|(v_h,\widetilde{u}_h)|^2\\
			&\ge0.
		\end{aligned}
	\end{equation}
	The proof is completed.
\end{proof}

\begin{theo}\label{theo3.3}
	Let $(u_h,\lambda_h)$ be the first eigenpair of \eqref{2.4} satisfying $(u_h,\widetilde{u}_h)\ge0$, then we have
	\begin{eqnarray}\label{3.8}
		\la (A_h^u-\widetilde{\lambda}_h)(u_h-\widetilde{u}_h),u_h-\widetilde{u}_h\ran\gtrsim h^\eps\|u_h-\widetilde{u}_h\|_{1,h}^2.
	\end{eqnarray}
\end{theo}
\begin{proof}
	By tkaing $v_h=u_h-\widetilde{u}_h$ in \eqref{3.7}, we obtain
	\begin{eqnarray*}
		\la (A_h^u-\widetilde{\lambda}_h)(u_h-\widetilde{u}_h),u_h-\widetilde{u}_h\ran
		\gtrsim\|u_h-\widetilde{u}_h\|^2-|(u_h-\widetilde{u}_h,\widetilde{u}_h)|^2.
	\end{eqnarray*}
	Since $\|u_h\|=\|\widetilde{u}_h\|=1$, it follows that
	\begin{eqnarray*}
		\|u_h-\widetilde{u}_h\|^2-|(u_h-\widetilde{u}_h,\widetilde{u}_h)|^2
		%		&=\|u_h-\widetilde{u}_h\|^2-|(u_h,\widetilde{u}_h)|^2+2(u_h,\widetilde{u}_h)-1\\
		&=1-|(u_h,\widetilde{u}_h)|^2\\
		&\ge1-(u_h,\widetilde{u}_h)\\
		&=\frac12\|u_h-\widetilde{u}_h\|^2.
	\end{eqnarray*}
	Therefore, we arrive at
	\begin{eqnarray}\label{3.9}
		\la (A_h^u-\widetilde{\lambda}_h)(u_h-\widetilde{u}_h),u_h-\widetilde{u}_h\ran
		\gtrsim\|u_h-\widetilde{u}_h\|^2.
	\end{eqnarray}
	A combination of Assumption (A2) and \eqref{3.22} leads to
	\begin{eqnarray*}
		\la (A_h^u-\widetilde{\lambda}_h)v_h,v_h\ran&=a_h(v_h,v_h)+(f(u^2)v_h,v_h)-\widetilde{\lambda}_h\|v_h\|^2\\
		&\ge h^\eps\|v_h\|_{1,h}^2+(f(0)-\widetilde{\lambda}_h)\|v_h\|^2,
	\end{eqnarray*}
	which implies that there exists $\eta>0$ such that
	\begin{eqnarray}\label{3.10}
		(A_h^u-\widetilde{\lambda}_h)v_h,v_h\ran\ge h^\eps\|v_h\|_{1,h}^2-\eta\|v_h\|^2,\quad\forall v_h\in V_h.
	\end{eqnarray}
	Combining \eqref{3.9}-\eqref{3.10} we demonstrate \eqref{3.8}. The proof is finished.
\end{proof}

%\begin{coro}\label{coro2}
%	$\la (E_h''(u)-\widetilde{\lambda}_h)(u_h-\widetilde{u}_h),u_h-\widetilde{u}_h\ran\gtrsim h^\eps\|u_h-\widetilde{u}_h\|_{1,h}^2$.
%\end{coro}
%\begin{proof}
%	By the definition \eqref{E''h} and Assumption (A2), we have
%	\begin{eqnarray*}
	%		\la (E_h''(u)-\widetilde{\lambda}_h)v_h,w_h\ran&=\la (A_h^u-\widetilde{\lambda}_h)v_h,w_h\ran+2(f'(u^2)u^2v_h,w_h)\\
	%		&\ge\la (A_h^u-\widetilde{\lambda}_h)v_h,w_h\ran.
	%	\end{eqnarray*}
%	which combines with Theorem \ref{theo3.3} finishes the proof.
%\end{proof}

\begin{coro}
	There holds the following estimate:
	\begin{eqnarray}\label{3.11}
		\la (E_h''(u)-\widetilde{\lambda}_h)v_h,v_h\ran\gtrsim h^\eps\|v_h\|_{1,h}^2,\quad\forall v_h\in V_h.
	\end{eqnarray}
\end{coro}
\begin{proof}
	It follows from the definition \eqref{E''h} and Assumption (A2) that
	\begin{eqnarray*}
		\la (E_h''(u)-\widetilde{\lambda}_h)v_h,v_h\ran&=\la (A_h^u-\widetilde{\lambda}_h)v_h,v_h\ran+2\int_\Omega f'(u^2)u^2v_h^2 d\Omega\\
		&\ge\la (A_h^u-\widetilde{\lambda}_h)v_h,v_h\ran,
	\end{eqnarray*}
	which together with \eqref{3.10} leads to
	\begin{eqnarray}\label{3.12}
		\la (E_h''(u)-\widetilde{\lambda}_h)v_h,v_h\ran \ge h^\eps\|v_h\|_{1,h}^2-\eta\|v_h\|^2,\quad\forall v_h\in V_h.
	\end{eqnarray}
	Reasoning by contraction, we deduce from Theorem \ref{theo2} that
	\begin{eqnarray}\label{3.13}
		\la (E_h''(u)-\widetilde{\lambda}_h)v_h,v_h\ran\gtrsim \|v_h\|^2,\quad\forall v_h\in V_h.
	\end{eqnarray}
	Combining \eqref{3.12}-\eqref{3.13} we demonstrate \eqref{3.11} and complete the proof.
\end{proof}

\begin{lemma}\label{theo4}
	Assume that the ground state $u\in H^{k+1}(\Omega)$, then under Assumptions (A1)-(A3), we have the following estimate:
	\begin{eqnarray*}
		\|u_h-\widetilde{u}_h\|_{1,h}=o(h^\eps),~as~h\rightarrow0.
	\end{eqnarray*}
\end{lemma}
\begin{proof}
	It follows from the definition \eqref{E} and \eqref{Eh} that
	\begin{eqnarray*}
		2(E_h(u_h)-E(u))&=a_h(u_h,u_h)+\int_\Omega F(u_h^2) d\Omega-a(u,u)-\int_\Omega F(u^2) d\Omega\\
		&=\la A_h^u u_h,u_h\ran -\la A^uu,u\ran+\int_\Omega F(u_h^2)-F(u^2)-f(u^2)(u_h^2-u^2) d\Omega.
	\end{eqnarray*}
	Then, a combination of Assumption (A2) and \eqref{3.3} implies that
	\begin{eqnarray*}
		2(E_h(u_h)-E(u))&\ge \la A_h^u u_h,u_h\ran -\la A^uu,u\ran\\
		&=\la (A_h^u-\widetilde{\lambda}_h) u_h,u_h\ran+\widetilde{\lambda}_h-\lambda\\
		&=\la (A_h^u-\widetilde{\lambda}_h) (u_h-\widetilde{u}_h),u_h-\widetilde{u}_h\ran+\widetilde{\lambda}_h-\lambda.
	\end{eqnarray*}
	Furthermore, by Theorem \ref{theo3.3} we arrive at
	\begin{eqnarray*}
		2(E_h(u_h)-E(u))\ge h^\eps\|u_h-\widetilde{u}_h\|_{1,h}^2+\widetilde{\lambda}_h-\lambda,
	\end{eqnarray*}
	which leads to
	\begin{eqnarray*}
		\|u_h-\widetilde{u}_h\|_{1,h}^2&\lesssim h^{-\eps}(E_h(u_h)-E(u))+h^{-\eps}(\lambda-\widetilde{\lambda}_h)\\
		&\lesssim h^{-\eps}(E_h(Q_hu)-E(u))+h^{2k-3\eps}.
	\end{eqnarray*}
	Next, we turn to estimate $E_h(Q_hu)-E(u)$. By the defnition \eqref{E} and \eqref{Eh}, we get
	\begin{equation}\label{3.14}
		\begin{aligned}
			2(E_h(Q_hu)-E(u))&=\|\nabla_wQ_hu\|^2-\|\nabla u\|^2+\|V^\frac{1}{2} Q_hu\|^2-\|V^\frac{1}{2} u\|^2\\
			&\quad+s(Q_hu,Q_hu)+\int_\Omega F((Q_hu)^2)-F(u^2)d\Omega.
		\end{aligned}
	\end{equation}
	From Lemma \ref{lem3} and \eqref{pro2}, the following estimate holds,
	\begin{equation}\label{3.15}
		\begin{aligned}
			\|\nabla_wQ_hu\|^2-\|\nabla u\|^2&=\|\Q_h\nabla u\|^2-\|\nabla u\|^2\\
			&=-\sumT\|\nabla u-\Q_h\nabla u\|_T^2\\
			&\lesssim h^{2k}\|u\|_{2k+1}^2.
		\end{aligned}
	\end{equation}
	Since $V\in L^2(\Omega)$ and $u\in L^\infty(\Omega)$, we have
	\begin{equation}\label{3.16}
		\begin{aligned}
			\left|\|V^\frac{1}{2} Q_hu\|^2-\|V^\frac{1}{2} u\|^2\right|&\lesssim \sumT\|\nabla(Q_hu-u)\|_T+\sumT\|Q_hu-u\|_T\\
			&\lesssim h^k.
		\end{aligned}
	\end{equation}
	By the trace inequality \eqref{trace} and \eqref{pro1}, we obtain
	\begin{equation}\label{3.17}
		\begin{aligned}
			s(Q_hu,Q_hu)&=\sumT h_T^{-1+\eps}\| Q_b(Q_0u-u)\|_{\partial T}^2\\
			&\le \sumT h_T^{-1+\eps}\| Q_0u-u\|_{\partial T}^2\\
			&\lesssim h^{2k+\eps}\|u\|_{k+1}^2.
		\end{aligned}
	\end{equation}
	Finally, we deduce from Assumption (A2) that
	\begin{equation}\label{3.18}
		\begin{aligned}
			\left|\int_\Omega F((Q_hu)^2)-F(u^2)d\Omega\right|&\lesssim \int_\Omega |(Q_hu)^2-u^2|d\Omega\\
			&\lesssim \sumT\|Q_hu-u\|_T\\
			&\lesssim h^{k+1}\|u\|_{k+1}.
		\end{aligned}
	\end{equation}
	Substituting \eqref{3.14}-\eqref{3.17} into \eqref{3.13}, it follows that
	\begin{eqnarray*}
		E_h(Q_hu)-E(u)\lesssim h^k,
	\end{eqnarray*}
	which implies that
	\begin{eqnarray*}
		E_h(Q_hu)-E(u)=o(h^{2\eps}),~ as~ h\rightarrow0.
	\end{eqnarray*}
	The proof is completed.
\end{proof}

\begin{theo}\label{theo5}
	Assume that the ground state $u\in H^{k+1}(\Omega)$, then under Assumptions (A1)-(A4), the following estimate holds true:
	\begin{eqnarray*}
		\lambda_h-\widetilde{\lambda}_h=o(h^{\eps}),~as~h\rightarrow0.
	\end{eqnarray*}
\end{theo}
\begin{proof}
	It follows from \eqref{1.3}, \eqref{2.4} and \eqref{3.3} that
	\begin{equation}\label{3.19}
		\begin{aligned}
			\lambda_h-\widetilde{\lambda}_h&=\la A_h^u u_h,u_h\ran-\widetilde{\lambda}_h+\int_\Omega (f(u_h^2)-f(u^2))u_h^2 d\Omega\\
			&=\la (A_h^u-\widetilde{\lambda}_h) (u_h-\widetilde{u}_h),u_h-\widetilde{u}_h\ran+\int_\Omega (f(u_h^2)-f(u^2))u_h^2 d\Omega.
		\end{aligned}
	\end{equation}
	Then, by Theorem \ref{theo1} we obtain
	\begin{eqnarray*}
		\Big|\la (A_h^u-\widetilde{\lambda}_h) (u_h-\widetilde{u}_h),u_h-\widetilde{u}_h\ran\Big|\lesssim\|u_h-\widetilde{u}_h\|_{1,h}^2.
	\end{eqnarray*}
	Moreover, from Assumption (A4) we demonstrate
	\begin{eqnarray*}
		\left|\int_\Omega (f(u_h^2)-f(u^2))u_h^2 d\Omega\right|&=\left|\int_\Omega f'(\xi_h^2)(u_h^2-u^2)u_h^2 d\Omega\right|\\
		&\lesssim \|u-u_h\|_{1,h}\\
		&\le\|u-\widetilde{u}_h\|_{1,h}+\|u_h-\widetilde{u}_h\|_{1,h}.
	\end{eqnarray*}
	Thus, combining \eqref{3.5} and Lemma \ref{theo4} and , we arrive at
	\begin{eqnarray*}
		|\lambda_h-\widetilde{\lambda}_h|\lesssim \|u-\widetilde{u}_h\|_{1,h}+\|u_h-\widetilde{u}_h\|_{1,h}=o(h^{\eps}),~ as ~h\rightarrow0.
	\end{eqnarray*}
	The proof is finished.
\end{proof}

\begin{lemma}
	Define $\widetilde{V}_h^\perp=\{v_h\in V_h:(v_h,\widetilde{u}_h)=0\}$. Then under Assumptions (A1)-(A5), we have the following estimate for each $w_h\in V_h$:
	\begin{equation}\label{3.20}
		\begin{aligned}
			h^\eps\frac{|(w_h,u_h-\widetilde{u}_h)|}{\|w_h\|}\lesssim& \|u-u_h\|^r_{L^{6r/(5-s)}}+|\lambda_h-\widetilde{\lambda}_h|\|u_h-\widetilde{u}_h\|\\
			\quad&+\|u-\widetilde{u}_h\|+\|u_h-\widetilde{u}_h\|^2.
		\end{aligned}
	\end{equation}
\end{lemma}
\begin{proof}
	Consider the following auxiliary problem: Find $\psi_{w_h}\in \widetilde{V}_h^\perp$ such that
	\begin{eqnarray*}
		\la (E_h''(u) -\widetilde{\lambda}_h)\psi_{w_h}, v_h\ran=(w_h,v_h),\quad\forall  v_h\in \widetilde{V}_h^\perp.
	\end{eqnarray*}
	Then, it follows from $\|u_h\|$=$\|\widetilde{u}_h\|$=1 that
	\begin{eqnarray*}
		(w_h,u_h-\widetilde{u}_h)&=(w_h,u_h-(u_h,\widetilde{u}_h)\widetilde{u}_h)-\frac12\|u_h-\widetilde{u}_h\|^2(w_h,\widetilde{u}_h)
		\\
		&=\la (E_h''(u) -\widetilde{\lambda}_h)\psi_{w_h}, u_h-(u_h,\widetilde{u}_h)\widetilde{u}_h\ran-\frac12\|u_h-\widetilde{u}_h\|^2(w_h,\widetilde{u}_h)
		\\
		&=\la (E_h''(u) -\widetilde{\lambda}_h)(u_h-\widetilde{u}_h),\psi_{w_h} \ran-\frac12\|u_h-\widetilde{u}_h\|^2\la (E_h''(u) -\widetilde{\lambda}_h)\widetilde{u}_h,\psi_{w_h}\ran
		\\
		&\quad-\frac12\|u_h-\widetilde{u}_h\|^2(w_h,\widetilde{u}_h)
		\\
		&=\la (E_h''(u) -\widetilde{\lambda}_h)(u_h-\widetilde{u}_h),\psi_{w_h} \ran-\|u_h-\widetilde{u}_h\|^2\int_\Omega f'(u^2)u^2\widetilde{u}_h\psi_{w_h}d\Omega
		\\
		&\quad-\frac12\|u_h-\widetilde{u}_h\|^2(w_h,\widetilde{u}_h).
	\end{eqnarray*}
	Furthermore, by \eqref{2.4} and \eqref{3.3}, we have
	\begin{eqnarray*}
		\la (E_h''(u) -\widetilde{\lambda}_h)v_h,\psi_{w_h} \ran&=\la (A_h^u -\widetilde{\lambda}_h)v_h,\psi_{w_h} \ran+2\int_\Omega f'(u^2)u^2\psi_{w_h}v_h d\Omega
		\\
		&=(\lambda_h -\widetilde{\lambda}_h)(v_h,\psi_{w_h})-\int_\Omega(f(u_h^2)-f(u^2))u_h^2\psi_{w_h} d\Omega
		\\
		&\quad+2\int_\Omega f'(u^2)u^2\psi_{w_h}v_h d\Omega
		\\
		&=(\lambda_h -\widetilde{\lambda}_h)(v_h,\psi_{w_h})+2\int_\Omega f'(u^2)u^2\psi_{w_h}(u-\widetilde{u}_h) d\Omega
		\\
		&\quad-\int_\Omega \left((f(u_h^2)-f(u^2))u_h^2-2f'(u^2)u^2(u_h-u)\right)\psi_{w_h} d\Omega,
	\end{eqnarray*}
	where $v_h=u_h-\widetilde{u}_h$.
	
	Thus, together with Assumptions (A4)-(A5) we obtain
	\begin{eqnarray*}
		|(w_h,u_h-\widetilde{u}_h)|&\le \|u_h-\widetilde{u}_h\|^2\|\psi_{w_h}\|+|\lambda_h -\widetilde{\lambda}_h|\|u_h-\widetilde{u}_h\|\|\psi_{w_h}\|+\|u-\widetilde{u}_h\|\|\psi_{w_h}\|
		\\
		&\quad+\|u-u_h\|_{L^{6r/(5-s)}}^r\|\psi_{w_h}\|_{1,h}+\|u_h-\widetilde{u}_h\|^2\|w_h\|.
	\end{eqnarray*}
	By \eqref{3.11} and Lemma \ref{lem6}, we demonstrate
	\begin{eqnarray*}
		\|\psi_{w_h}\|\lesssim\|\psi_{w_h}\|_{1,h}\lesssim h^{-\eps}\|w_h\|,
	\end{eqnarray*}
	which leads to \eqref{3.7} and finishes the proof.	
\end{proof}

\begin{lemma}\label{coro3.3}
	Assume that the ground state $u\in H^{k+1}(\Omega)$, then under Assumptions (A1)-(A5), there holds the following estimate:
	\begin{eqnarray*}
		h^\eps\|u_h-\widetilde{u}_h\|\lesssim \|u_h-\widetilde{u}_h\|^r_{L^{6r/(5-s)}}+\|u-\widetilde{u}_h\|^r_{L^{6r/(5-s)}}+\|u-\widetilde{u}_h\|.
	\end{eqnarray*}
\end{lemma}
\begin{proof}
	By Taking $w_h=u_h-\widetilde{u}_h$ in \eqref{3.7}, we have
	\begin{equation*}
		\begin{aligned}
			h^\eps\|u_h-\widetilde{u}_h\|\lesssim \|u-u_h\|^r_{L^{6r/(5-s)}}+|\lambda_h-\widetilde{\lambda}_h|\|u_h-\widetilde{u}_h\|+\|u-\widetilde{u}_h\|+\|u_h-\widetilde{u}_h\|^2.
		\end{aligned}
	\end{equation*}
	Then, it follows from Lemma \ref{theo4} and Theorem \ref{theo5} that
	\begin{eqnarray*}
		h^\eps\|u_h-\widetilde{u}_h\|&\lesssim \|u-u_h\|^r_{L^{6r/(5-s)}}+\|u-\widetilde{u}_h\|\\
		&\le\|u_h-\widetilde{u}_h\|^r_{L^{6r/(5-s)}}+\|u-\widetilde{u}_h\|^r_{L^{6r/(5-s)}}+\|u-\widetilde{u}_h\|.
	\end{eqnarray*}
	The proof is completed.
\end{proof}

Then, we are ready to give the $L^2$ and $H^1$ error estimate for the ground state.
\begin{theo}\label{theo7}
	Let $(u_h,\lambda_h)$ be the first eigenpair of \eqref{2.4}. Assume that the ground state $u\in H^{k+1}(\Omega)$, then under Assumptions (A1)-(A5), the following estimate holds true:
	\begin{eqnarray*}
		%		\|u_h-\widetilde{u}_h\|\lesssim h^{-\eps}\|u-\widetilde{u}_h\|.
		\|u-u_h\|\lesssim h^{k+1-2\eps}.
	\end{eqnarray*}
\end{theo}
\begin{proof}
	Since $r+s\le 3$, by the Young's inequality and Lemma \ref{lem6}, we obtain
	\begin{equation}\label{3.21}
		\begin{aligned}
			\|u_h-\widetilde{u}_h\|^r_{L^{6r/(5-s)}}&\le \|u_h-\widetilde{u}_h\|^{(5-r-s)/2}\|u_h-\widetilde{u}_h\|^{(3r-5+s)/2}_{L^6}\\
			&\le \|u_h-\widetilde{u}_h\|\|u_h-\widetilde{u}_h\|_{1,h}^{r-1}.
		\end{aligned}
	\end{equation}
	Similarly, we have
	\begin{eqnarray*}
		\|u-\widetilde{u}_h\|^r_{L^{6r/(5-s)}}\lesssim \|u-\widetilde{u}_h\|\|u-\widetilde{u}_h\|_{1,h}^{r-1}.
	\end{eqnarray*}
	
	Then, it follows from Lemma \ref{coro3.3} that
	\begin{eqnarray*}
		h^\eps\|u_h-\widetilde{u}_h\|\lesssim \|u_h-\widetilde{u}_h\|\|u_h-\widetilde{u}_h\|_{1,h}^{r-1}+\|u-\widetilde{u}_h\|\|u-\widetilde{u}_h\|_{1,h}^{r-1}+\|u-\widetilde{u}_h\|,
	\end{eqnarray*}
	which implies that
	\begin{eqnarray*}
		\|u_h-\widetilde{u}_h\|\lesssim h^{-\eps}\|u-\widetilde{u}_h\|.
	\end{eqnarray*}
	
	Finally, combining \eqref{3.6} and Lemma \ref{theo7}, we demonstrate
	\begin{eqnarray*}
		\|u-u_h\|\le \|u-\widetilde{u}_h\|+\|u_h-\widetilde{u}_h\|\lesssim h^{-\eps}\|u-\widetilde{u}_h\|\lesssim h^{k+1-2\eps},
	\end{eqnarray*}
	which finishes the proof.
\end{proof}

\begin{theo}\label{lem9}
	Let $(u_h,\lambda_h)$ be the first eigenpair of \eqref{2.4}. Under the conditions of Theorem \ref{theo7}, we have
	\begin{eqnarray*}
		\|u-u_h\|_{1,h}\lesssim h^{k-\eps}.
		%		\|u_h-\widetilde{u}_h\|_{1,h}\lesssim h^{k+1-2\eps}.
	\end{eqnarray*}
\end{theo}
\begin{proof}
	It follows from \eqref{3.3} and \eqref{3.11} that
	\begin{eqnarray*}
		h^\eps\|u_h-\widetilde{u}_h\|_{1,h}^2&\lesssim \la (E_h''(u)-\widetilde{\lambda}_h)(u_h-\widetilde{u}_h),u_h-\widetilde{u}_h\ran\\
		&=\la (A_h^u-\widetilde{\lambda}_h)(u_h-\widetilde{u}_h),u_h-\widetilde{u}_h\ran+2\int_\Omega f'(u^2)u^2(u_h-\widetilde{u}_h)^2d\Omega\\
		&=\la (A_h^u-\widetilde{\lambda}_h)u_h,u_h-\widetilde{u}_h\ran+\frac12(\lambda-\widetilde{\lambda}_h)\|u_h-\widetilde{u}_h\|^2\\
		&\quad+2\int_\Omega f'(u^2)u^2(u_h-\widetilde{u}_h)^2d\Omega\\
		&=2\int_\Omega \left((f(u^2)-f(u_h^2))u_h+f'(u^2)u^2(u_h-\widetilde{u}_h)\right)(u_h-\widetilde{u}_h)d\Omega\\
		&\quad+\frac12(\lambda-\widetilde{\lambda}_h)\|u_h-\widetilde{u}_h\|^2.
	\end{eqnarray*}
	
	Then, by Assumptions (A4)-(A5) we get
	\begin{eqnarray*}
		h^\eps\|u_h-\widetilde{u}_h\|_{1,h}^2&\lesssim \int_\Omega \left((f(u^2)-f(u_h^2))u_h+f'(u^2)u^2(u_h-u)\right)(u_h-\widetilde{u}_h)d\Omega\\
		&\quad+\|u-\widetilde{u}_h\|\|u_h-\widetilde{u}_h\|+(\lambda-\widetilde{\lambda}_h)\|u_h-\widetilde{u}_h\|^2\\
		&\lesssim \|u-u_h\|_{L^{6r/(5-s)}}^r\|u_h-\widetilde{u}_h\|_{1,h}+\|u-\widetilde{u}_h\|\|u_h-\widetilde{u}_h\|\\
		&\quad+(\lambda-\widetilde{\lambda}_h)\|u_h-\widetilde{u}_h\|^2.
	\end{eqnarray*}
	Furthermore, from \eqref{3.5}, \eqref{3.21} and Lemma \ref{lem6}, we have
	\begin{eqnarray*}
		h^\eps\|u_h-\widetilde{u}_h\|_{1,h}^2&\lesssim \|u-u_h\|\|u-u_h\|_{1,h}^{r-1}\|u_h-\widetilde{u}_h\|_{1,h}+\|u-\widetilde{u}_h\|\|u_h-\widetilde{u}_h\|\\
		&\lesssim\|u-u_h\|\|u-u_h\|_{1,h}^{r-1}\|u_h-\widetilde{u}_h\|_{1,h}+\|u-\widetilde{u}_h\|\|u_h-\widetilde{u}_h\|_{1,h},
	\end{eqnarray*}
	which implies that
	\begin{eqnarray*}
		h^\eps\|u_h-\widetilde{u}_h\|_{1,h}&\lesssim\|u-u_h\|\|u-u_h\|_{1,h}^{r-1}+\|u-\widetilde{u}_h\|\\
		&\le\|u-u_h\|\left(\|u_h-\widetilde{u}_h\|_{1,h}^{r-1}+\|u-\widetilde{u}_h\|_{1,h}^{r-1}\right)+\|u-\widetilde{u}_h\|.
	\end{eqnarray*}
	
	Finally, since $r>1$, by \eqref{3.5}, Theorem \ref{theo4} and Theorem \ref{theo7}, we arrive at
	\begin{eqnarray*}
		\|u_h-\widetilde{u}_h\|_{1,h}\lesssim h^{-\eps}(\|u-u_h\|+\|u-\widetilde{u}_h\|)
		\lesssim h^{k+1-2\eps},
	\end{eqnarray*}
	which together with \eqref{3.5} leads to
	\begin{eqnarray*}
		\|u-u_h\|_{1,h}\le \|u-\widetilde{u}_h\|_{1,h}+\|u_h-\widetilde{u}_h\|_{1,h}\lesssim h^{k-\eps}.
	\end{eqnarray*}
	The proof is completed.
\end{proof}

%\begin{theo}
%	$|\lambda_h-\widetilde{\lambda}_h|\lesssim \|u_h-\widetilde{u}_h\|_{1,h}$.
%\end{theo}

Next, we turn to the estimation of the ground state eigenvalue.
\begin{theo}\label{theo9}
	Let $(u_h,\lambda_h)$ be the first eigenpair of \eqref{2.4}. Under the conditions of Theorem \ref{theo7}, there we have
	\begin{eqnarray*}
		|\lambda-\lambda_h|\lesssim h^{k-\eps}.
	\end{eqnarray*}
\end{theo}
\begin{proof}
	It follows from (34) in \cite{Eric2009} that
	\begin{eqnarray*}
		\left|\int_\Omega (f(u_h^2)-f(u^2))u_h^2 d\Omega\right|\lesssim\|u-u_h\|_{L^{6/(5-2q)}}.
	\end{eqnarray*}
	Then, by Theorem \ref{theo1}, Lemma \ref{lem6} and \eqref{3.19} we arrive at
	\begin{eqnarray*}
		|\lambda_h-\widetilde{\lambda}_h|&\le|\la (A_h^u-\widetilde{\lambda}_h)(u_h-\widetilde{u}_h),u_h-\widetilde{u}_h\ran|+\left|\int_\Omega (f(u_h^2)-f(u^2))u_h^2 d\Omega\right|
		\\
		&\lesssim \|u_h-\widetilde{u}_h\|_{1,h}^2+\|u-u_h\|_{L^{6/(5-2q)}}
		\\
		&\lesssim \|u_h-\widetilde{u}_h\|_{1,h}^2+\|u-u_h\|_{1,h}.
	\end{eqnarray*}
	Thus, from \eqref{3.4}, Lemma \ref{lem9} and Theorem \ref{theo9}, we obtain
	\begin{eqnarray*}
		|\lambda-\lambda_h|&\le|\lambda-\widetilde{\lambda}_h|+|\lambda_h-\widetilde{\lambda}_h|
		\\
		&\lesssim|\lambda-\widetilde{\lambda}_h|+\|u_h-\widetilde{u}_h\|_{1,h}^2+\|u-u_h\|_{1,h}
		\\
		&\lesssim h^{k-\eps}.
	\end{eqnarray*}
	The proof is finished.
\end{proof}

Finally, the following result shows the WG approximation gives asymptotic lower bound for the ground state energy.
\begin{theo}\label{theo10}
	Under the conditions of Theorem \ref{theo7}, we have
	\begin{eqnarray*}
		0\le E(u)-E_h(u_h)\lesssim h^{2k-2\eps},
	\end{eqnarray*}
	when h is sufficiently small.
\end{theo}
\begin{proof}
	From \eqref{E}, \eqref{Eh} and \eqref{3.3}, we derive
	\begin{equation}\label{3.22}
		\begin{aligned}
			2(E_h(u_h)-E(u))&=\la (A_h^u-\widetilde{\lambda}_h) (u_h-\widetilde{u}_h),u_h-\widetilde{u}_h\ran+\widetilde{\lambda}_h-\lambda\\
			&\quad+\int_\Omega F(u_h^2)-F(u^2)-f(u^2)(u_h^2-u^2) d\Omega.
		\end{aligned}
	\end{equation}
	Then, a combination of Theorem \ref{theo2} and Lemma \ref{lem9} leads to
	\begin{eqnarray}\label{3.23}
		0\le\la (A_h^u-\widetilde{\lambda}_h) (u_h-\widetilde{u}_h),u_h-\widetilde{u}_h\ran\lesssim h^{2k+2-4\eps}.
	\end{eqnarray}
	Since $f'(t)\ge0$, by Theorem \ref{theo7} we obtain
	\begin{eqnarray}\label{3.24}
		0\le\int_\Omega F(u_h^2)-F(u^2)-f(u^2)(u_h^2-u^2) d\Omega\lesssim \|u-u_h\|^2\lesssim h^{2k+2-4\eps}.
	\end{eqnarray}
	Finally, substituting \eqref{3.4}, \eqref{3.23}-\eqref{3.24} into \eqref{3.22}, we get
	\begin{eqnarray*}
		0\le h^{2k}+O(h^{2k+2-4\eps})\lesssim E(u)-E_h(u_h)\lesssim h^{2k-2\eps},
	\end{eqnarray*}
	when $h$ is sufficiently small. The proof is finished.
\end{proof}

\section{Lower bound and upper bound for the eigenvalue}
In this section, we apply the post-processing technique to give the lower bound and upper bound for the ground state eigenvalue of the GPE type, i.e. $F(v^2)=\frac{\beta}{2}v^4$ in \eqref{E}.

\begin{lemma}\label{lem10}
	Let $(\lambda_{1,h},u_{1,h})$ and $(\lambda_{2,h},u_{2,h})$ be the linear and quadratic Lagrange conforming element approximations for \eqref{1.3}, respectively. Then, there holds the following estimates
	\begin{eqnarray*}
		&|\lambda-\lambda_{,h}|\lesssim h^2,\quad \|u-u_{1,h}\|_1\lesssim h,\quad \|u-u_{1,h}\|_1^2\lesssim E_{1,h}-E\lesssim h^2,\\
		&|\lambda-\lambda_{2,h}|\lesssim h^4,\quad \|u-u_{2,h}\|_1\lesssim h^2,\quad \|u-u_{2,h}\|_1^2\lesssim E_{2,h}-E\lesssim h^4,
		%	 	|\lambda-\lambda_{j,h}|\lesssim h^{2j},~ \|u-u_{j,h}\|_1\lesssim h^j,~ \|u-u_{j,h}\|_1^2\lesssim E-E_{1,h}\lesssim h^{2j},\quad j=1,2,
	\end{eqnarray*}
	where
	\begin{eqnarray*}
		E_{j,h}=\frac12 a(u_{j,h},u_{j,h})+\frac{\beta}{4}\|u_{j,h}\|_{L^4},\quad j=1,2.
	\end{eqnarray*}
\end{lemma}
\begin{proof}
	See Theorem 3 in \cite{Eric2009}. 
\end{proof}

Let $(\lambda_h,u_h)$ by the WG approximation with $k=1$  for \eqref{1.3}. From \eqref{E} and \eqref{1.3}, it is easy to check the ground state $(\lambda,u)$ and $E$ satisfy
\begin{eqnarray*}
	\lambda=2E+\frac{\beta}{2}\|u\|_{L^4}^4.
\end{eqnarray*}
Similarly, we have
\begin{eqnarray*}
	\lambda_h=2E_h+\frac{\beta}{2}\|u_h\|_{L^4}^4,\quad
	\lambda_{j,h}=2E_{j,h}+\frac{\beta}{2}\|u_{j,h}\|_{L^4}^4,\quad j=1,2.
\end{eqnarray*}
Thus, we obtain
\begin{eqnarray}
	\lambda-\lambda_h&=2(E-E_h)+\frac{\beta}{2}(\|u\|_{L^4}^4-\|u_h\|_{L^4}^4),\label{4.1}\\
	\lambda_{j,h}-\lambda&=2(E_{j,h}-E)+\frac{\beta}{2}(\|u_{j,h}\|_{L^4}^4-\|u\|_{L^4}^4), \quad j=1,2.\label{4.2}
\end{eqnarray}
\begin{remark}
	Combining Lemma \ref{lem10}, Lemma \ref{lem11} and \ref{4.2}, we can derive
	\begin{eqnarray}\label{4.3}
		\frac{\beta}{2}\big|\|u_{j,h}\|_{L^4}^4-\|u\|_{L^4}^4\big|\lesssim h^{2j},\quad j=1,2.
	\end{eqnarray}
\end{remark}

Now, we are ready to give the lower bound for the ground state energy.
\begin{theo}
	Let $(\lambda_h,u_h)$ and $(\lambda_{2,h},u_{2,h})$ be the WG approximation with $k=1$ and quadratic Lagrange conforming element approximation for \eqref{1.3}. Then, when $h$ is sufficiently small, we have
	\begin{eqnarray*}
		0\le\lambda-\underline{\lambda}_h\lesssim h^{2-2\eps},
	\end{eqnarray*}
	where $\underline{\lambda}_h=\lambda_h+\frac{\beta}{2}(\|u_{2,h}\|_{L^4}-\|u_h\|_{L^4})$.
\end{theo}
\begin{proof}
	It follows from \eqref{4.1} that
	\begin{eqnarray*}
		\lambda-\lambda_h&=2(E-E_h)+\frac{\beta}{2}(\|u\|_{L^4}^4-\|u_h\|_{L^4}^4)
		\\
		&=2(E-E_h)+\frac{\beta}{2}(\|u\|_{L^4}^4-\|u_{2,h}\|_{L^4}^4)+\frac{\beta}{2}(\|u_{2,h}\|_{L^4}^4-\|u_h\|_{L^4}^4),
	\end{eqnarray*}
	which leads to
	\begin{eqnarray*}
		\lambda-\underline{\lambda}_h= 2(E-E_h)+\frac{\beta}{2}(\|u\|_{L^4}^4-\|u_{2,h}\|_{L^4}^4).
	\end{eqnarray*}
	By Theorem \ref{theo10} and \eqref{4.3} we obtain
	\begin{eqnarray*}
		h^2\lesssim E-E_h\lesssim h^{2-2\eps},\quad \frac{\beta}{2}\big|\|u\|_{L^4}^4-\|u_{2,h}\|_{L^4}^4\big|\lesssim h^4,
	\end{eqnarray*}
	which implies
	\begin{eqnarray*}
		0\le\lambda-\underline{\lambda}_h\lesssim h^{2-2\eps},
	\end{eqnarray*}
	when $h$ is sufficiently small. The proof is completed. 
\end{proof}

Next, we turn to the upper bound of the ground state eigenvalue. The following lower bound estimate is crucial in the analysis.
%Moreover, we have the following lower bound estimate for the energy.
\begin{lemma}\label{lem11}
	Let $(\lambda_{1,h},u_{1,h})$ be the linear Lagrange conforming element approximations for \eqref{1.3}. Then, we have
	\begin{eqnarray*}
		E_{1,h}-E\gtrsim h^{2}.
	\end{eqnarray*}
\end{lemma}
\begin{proof}
	It follows from Theorem 2.1 in \cite{Lin2013c} that
	\begin{eqnarray*}
		\|u-u_{1,h}\|_1\lesssim h.
	\end{eqnarray*}
	Thus, we demonstrate
	\begin{eqnarray*}
		E_{1,h}-E\gtrsim \|u-u_{1,h}\|_1^2\gtrsim h^{2},
	\end{eqnarray*}
	which finished the proof.
\end{proof}

Then, we have the following result.
\begin{theo}
	Let $(\lambda_{1,h},u_{1,h})$ and $(\lambda_{2,h},u_{2,h})$ be the linear and quadratic Lagrange conforming element approximations for \eqref{1.3}, respectively. Then, when $h$ is sufficiently small, we have
	\begin{eqnarray*}
		0\le\overline{\lambda}_h-\lambda\lesssim h^2,
	\end{eqnarray*}
	where $\overline{\lambda}_h=\lambda_{1,h}+\frac{\beta}{2}(\|u_{2,h}\|_{L^4}-\|u_{1,h}\|_{L^4})$.
\end{theo}
\begin{proof}
	It follows from \eqref{4.2} that
	\begin{eqnarray*}
		\lambda_{1,h}-\lambda&=2(E_{1,h}-E)+\frac{\beta}{2}(\|u_{1,h}\|_{L^4}^4-\|u\|_{L^4}^4)
		\\
		&=2(E_{1,h}-E)+\frac{\beta}{2}(\|u_{2,h}\|_{L^4}^4-\|u_h\|_{L^4}^4)-\frac{\beta}{2}(\|u_{2,h}\|_{L^4}^4-\|u_{1,h}\|_{L^4}^4),
	\end{eqnarray*}
	which leads to
	\begin{eqnarray*}
		\overline{\lambda}_h-\lambda= 2(E_{1,h}-E)+\frac{\beta}{2}(\|u\|_{L^4}^4-\|u_{2,h}\|_{L^4}^4).
	\end{eqnarray*}
	By Lemma \ref{lem10}, Lemma \ref{lem11} and \eqref{4.2} we obtain
	\begin{eqnarray*}
		h^2\lesssim E_{1,h}-E\lesssim h^2,\quad \frac{\beta}{2}(\|u_{2,h}\|_{L^4}^4-\|u\|_{L^4}^4)\lesssim h^4,
	\end{eqnarray*}
	which implies
	\begin{eqnarray*}
		0\le\overline{\lambda}_h-\lambda\lesssim h^2,
	\end{eqnarray*}
	when $h$ is sufficiently small. The proof is completed. 
\end{proof}

\section{Numerical experiments}In this section, we present some numerical examples to check the efficiency and lower bound property for the WG scheme \eqref{2.4}. In the following experiments, uniform triangular mesh is applied (see FIG \ref{pic1}).

\begin{figure}[ht]
	\centering
	\includegraphics[width=0.3\textwidth]{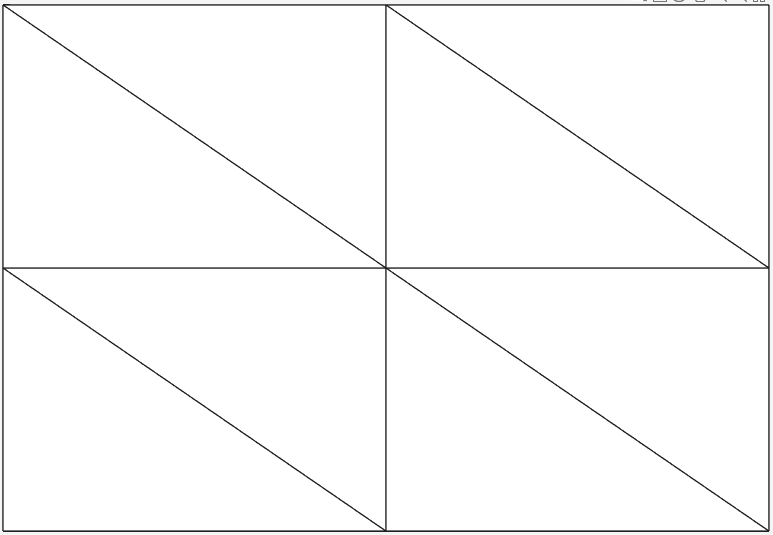}\quad
	\includegraphics[width=0.3\textwidth]{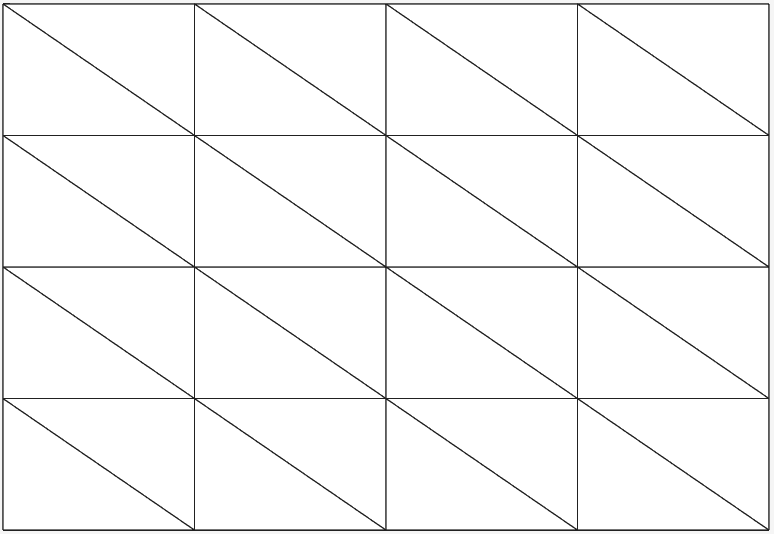}\quad
	\includegraphics[width=0.3\textwidth]{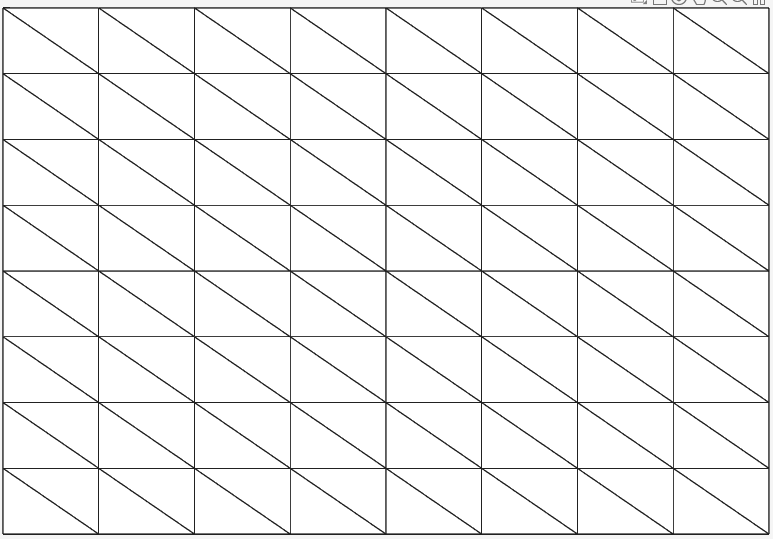}
	\caption{Uniform triangular mesh for unit square domain.}\label{pic1}
\end{figure}

Since the exact ground states are unknown, we compute the convergence rate by the following estimate
\begin{eqnarray*}
	Order\approx\lg\left(\frac{\lambda_h-\lambda_{\frac{h}{2}}}{\lambda_{\frac{h}{2}}-\lambda_{\frac{h}{4}}}\right)/\lg2.
\end{eqnarray*}

\begin{example}
	Consider the nonlinear eigenvalue problem \eqref{1.3} in the unit square domain $\Omega=(-8,8)^2$ with $V(x,y)=1$ and $f(u^2)=u^2$. The corresponding results are shown in Tables \ref{t1} -\ref{t2}.
\end{example}
\begin{table}[ht]
	\centering
	\caption{WG approximations with $k=1$.}
	\label{t1}
	\begin{tabular}{|c|c|c|c|c|c|}\hline
		$N$&16&32&64&128&Order\\ \hline\hline
		$\lambda$& 1.084390 & 1.085393 & 1.085648 & 1.085712 & \textbf{1.99} 
		\\ \hline
		$E$& 0.540057 & 0.540555 & 0.540681 & 0.540712 & \textbf{1.99} 
		\\ \hline
	\end{tabular}
\end{table}
\begin{table}[ht]
	\centering
	\caption{WG approximations with $k=2$.}
	\label{t2}
	\begin{tabular}{|c|c|c|c|c|}\hline
		$N$&16&32&64&Order\\ \hline\hline
		$\lambda$& 1.085730 & 1.085733 & 1.085733 & \textbf{4.01}
		\\ \hline
		$E$& 0.540722 & 0.540723 & 0.540723 & \textbf{4.01} 
		\\ \hline
	\end{tabular}
\end{table}

\begin{example}
	Consider the nonlinear eigenvalue problem \eqref{1.3} in the unit square domain $\Omega=(-1,1)^2$ with $V(x,y)=x^2+y^2$ and $f(u^2)=2 u^2$. The corresponding results are shown in Tables \ref{t3} -\ref{t4}.
\end{example}

\begin{table}[ht]
	\centering
	\caption{WG approximations with $k=1$.}
	\label{t3}
	\begin{tabular}{|c|c|c|c|c|c|}\hline
		$N$&16&32&64&128&Order\\ \hline\hline
		$\lambda$& 6.213447 & 6.278429 & 6.294926 & 6.299067 & \textbf{1.99} 
		\\ \hline
		$E$& 2.834190 & 2.866055 & 2.874147 & 2.876178 & \textbf{1.99}
		\\ \hline
	\end{tabular}
\end{table}
\begin{table}[ht]
	\centering
	\caption{WG approximations with $k=2$.}
	\label{t4}
	\begin{tabular}{|c|c|c|c|c|}\hline
		$N$&16&32&64&Order\\ \hline\hline
		$\lambda$& 6.300245 & 6.300435 & 6.300447 & \textbf{4.01} 
		\\ \hline
		$E$&2.876763 & 2.876850 & 2.876855 & \textbf{4.01} 
		\\ \hline
	\end{tabular}
\end{table}

\begin{example}
	Consider the nonlinear eigenvalue problem \eqref{1.3} in the unit square domain $\Omega=(-6,6)^2$ with $V(x,y)=x^2+y^2$ and $f(u^2)=\beta u^2$. We set $\beta$=20, 200 and 2000. The corresponding results are shown in Tables \ref{t5} -\ref{t6}.
\end{example}

\begin{table}[ht]
	\centering
	\caption{WG approximations with $k=1$.}
	\label{t5}
	\begin{tabular}{|c|c|c|c|c|c|}\hline
		$N$&16&32&64&128&Order\\ \hline\hline
		\multicolumn{6}{|c|}{$\beta=\textbf{20}$}\\ \hline
		$\lambda$& 4.075515 & 4.112630 & 4.123643 & 4.126529 & \textbf{1.93} 
		\\ \hline
		$E$&1.544138 & 1.578598 & 1.588763 & 1.591422 & \textbf{1.93}
		\\ \hline
		\multicolumn{6}{|c|}{$\beta=\textbf{200}$}\\ \hline
		$\lambda$& 11.502264 & 11.514207 & 11.518100 & 11.519150 & \textbf{1.89} 
		\\ \hline
		$E$&3.929779 & 3.941214 & 3.944705 & 3.945630 & \textbf{1.92} 
		\\ \hline
		\multicolumn{6}{|c|}{$\beta=\textbf{200}$}\\ \hline
		$\lambda$& 35.923572 & 35.944136 & 35.951404 & 35.953439 & \textbf{1.84} 
		\\ \hline
		$E$&11.980593 & 11.984349 & 11.985595 & 11.985935 & \textbf{1.87} 
		\\ \hline
	\end{tabular}
\end{table}
\begin{table}[ht]
	\centering
	\caption{WG approximations with $k=2$ and $\beta=20$.}
	\label{t6}
	\begin{tabular}{|c|c|c|c|c|}\hline
		$N$&16&32&64&Order\\ \hline\hline
		\multicolumn{5}{|c|}{$\beta=\textbf{20}$}\\ \hline
		$\lambda$& 4.125947 & 4.127392 & 4.127496 & \textbf{3.78} 
		\\ \hline
		$E$& 1.590830 & 1.592214 & 1.592312 & \textbf{3.82} 
		\\ \hline
		\multicolumn{5}{|c|}{$\beta=\textbf{200}$}\\ \hline
		$\lambda$& 11.518527 & 11.519441 & 11.519503 & \textbf{3.89} 
		\\ \hline
		$E$&3.945356 & 3.945898 & 3.945941 & \textbf{3.65} 
		\\ \hline
	\end{tabular}
\end{table}

\begin{example}
	Consider the nonlinear eigenvalue problem \eqref{1.3} in the unit square domain $\Omega=(-8,8)^2$ with $V(x,y)=x^2+y^2+8e^{-(x-1)^2-y^2}$ and $f(u^2)=400 u^2$. The corresponding results are shown in Tables \ref{t7}.
\end{example}

\begin{table}[ht!]
	\centering
	\caption{WG approximations with $k=1$.}
	\label{t7}
	\begin{tabular}{|c|c|c|c|c|c|}\hline
		$N$&16&32&64&128&Order\\ \hline\hline
		$\lambda$& 16.613239 & 16.624031 & 16.628343 & 16.629569 & \textbf{1.81} 
		\\ \hline
		$E$& 5.830354 & 5.843679 & 5.848679 & 5.850096 & \textbf{1.82}
		\\ \hline
	\end{tabular}
\end{table}

\section*{Statements and Declarations}
%%%%%%%%%%%%%%%%%%%%%%%%%%%%%%%%%%%%%%%%%%%%%%%%%%%%%%%%%%%%%

\smallskip
\noindent
\textbf{Funding}.
This work was partly supported by the Beijing Natural Science Foundation (Grant No. Z200003), by the National Natural Science Foundation of China (Grant Nos. 12331015, 12301475, 12271208 and 11901015), by the National Center for Mathematics and Interdisciplinary Science, Chinese Academy of Sciences.

\smallskip
\noindent
\textbf{Data Availability}.
The code used in this work is available from the corresponding author on reasonable request.

\smallskip
\noindent
\textbf{Competing Interests}.
The authors declare that they have no relevant financial or non-financial interests to disclose.
%============================================================================

\bibliography{library}
\bibliographystyle{siam}

\newpage

\end{document}